\documentclass[reqno,11pt]{amsart}
\usepackage{amsmath, amsthm, amssymb, geometry, enumerate, latexsym}
\usepackage[all]{xy}
\usepackage{braket}
\usepackage{mathtools}
\def\al{\alpha}

\def\ep{\varepsilon}

\def\si{\sigma}

\def\vph{\varphi}

\def\Om{\Omega}

\def\cV{{\mathcal V}}

\def\ang#1{{\langle #1 \rangle}}

\def\P{{\mathbb P}}

\def\Z{{\mathbb Z}}
\def\N{{\mathbb N}}
\def\m{{\mathfrak m}}

\def\dim{\operatorname{dim}}

\def\Ext{\operatorname{Ext}}

\def\H{\operatorname{H}}
\def\Hom{\operatorname{Hom}}

\def\grmod{\mathsf{grmod}\,}
\def\GrMod{\mathsf{GrMod}\,}
\def\CM{\mathsf{CM^{\Z}}}

\def\uCM{\mathsf{\underline{CM}^{\Z}}}

\def\Db{\mathsf{D^b}}

\def\mod{\mathsf{mod}\,}
\def\Mod{\mathsf{Mod}\,}

\def\rnum#1{\expandafter{\romannumeral #1}}
\def\Rnum#1{\uppercase\expandafter{\romannumeral #1}}

\theoremstyle{plain} 
\newtheorem{thm}{Theorem}[section]
\newtheorem{cor}[thm]{Corollary}
\newtheorem*{thm*}{Theorem}
\newtheorem*{cor*}{Corollary}
\newtheorem{lem}[thm]{Lemma}
\newtheorem{prop}[thm]{Proposition}

\theoremstyle{definition}
\newtheorem{dfn}[thm]{Definition}
\newtheorem{ex}[thm]{Example}

\newtheorem{conj}[thm]{Conjecture}

\theoremstyle{remark}

\newtheorem{rem}[thm]{Remark}

\newcommand{\thmref}[1]{Theorem~\ref{#1}}
\newcommand{\lemref}[1]{Lemma~\ref{#1}}
\newcommand{\corref}[1]{Corollary~\ref{#1}}
\newcommand{\propref}[1]{Proposition~\ref{#1}}

\newcommand{\conjref}[1]{Conjecture~\ref{#1}}

\numberwithin{equation}{section}

\begin{document}

\title{On Kn\"orrer periodicity for quadric hypersurfaces in skew projective spaces}
\author{Kenta Ueyama}
\address{
Department of Mathematics, 
Faculty of Education,
Hirosaki University, 
1 Bunkyocho, Hirosaki, Aomori 036-8560, Japan}
\email{k-ueyama@hirosaki-u.ac.jp} 
\thanks{The author was supported by JSPS Grant-in-Aid for Early-Career Scientists 18K13381.}


\begin{abstract}
We study the structure of the stable category $\mathsf{\underline{CM}}^{\mathbb Z}(S/(f))$ of graded maximal Cohen-Macaulay module over $S/(f)$
where $S$ is a graded ($\pm 1$)-skew polynomial algebra in $n$ variables of degree 1, and $f =x_1^2 + \cdots +x_n^2$.
If $S$ is commutative, then the structure of $\mathsf{\underline{CM}}^{\mathbb Z}(S/(f))$ is well-known by Kn\"orrer's periodicity theorem.
In this paper, we prove  that if $n\leq 5$, then the structure of $\mathsf{\underline{CM}}^{\mathbb Z}(S/(f))$ is determined by
the number of irreducible components of the point scheme of $S$ which are isomorphic to ${\mathbb P}^1$.
\end{abstract}

\maketitle

\section{Introduction}

Throughout this paper, we fix an algebraically closed field $k$ of characteristic 0.

Kn\"orrer's periodicity theorem (\cite[Theorem 3.1]{Kn}) plays an essential role in Cohen-Macaulay representation theory of Gorenstein rings.
As a special case of Kn\"orrer's periodicity theorem, the following result is well-known (see also \cite{BEH}).
\begin{thm}\label{thm.Kn}
Let $S=k[x_1, \dots, x_n]$ be a graded polynomial algebra generated in degree 1, and $f=x_1^2 + x_2^2 + \cdots + x_n^2$.
Let $\uCM(S/(f))$ denote the stable category of graded maximal Cohen-Macaulay module over $S/(f)$.
\begin{enumerate}
\item If $n$ is odd, then $\uCM(S/(f)) \cong \uCM(k[x]/(x^2)) \cong \Db(\mod k)$.
\item If $n$ is even, then $\uCM(S/(f)) \cong \uCM(k[x,y]/(x^2+y^2)) \cong \Db(\mod k^2)$.
\end{enumerate}
\end{thm}

The purpose of this paper is to study a ``($\pm 1$)-skew'' version of \thmref{thm.Kn}.

\begin{dfn}
Let $n \in \N^+$.
\begin{enumerate}
\item We say that $S$ is a graded skew polynomial algebra if
\[S = k\ang{x_1, \dots, x_n}/(x_ix_j-\al_{ij}x_jx_i)_{1\leq i, j\leq n}\]
where $\al_{ii} = 1$ for every $1\leq i\leq n$, $\al_{ij}\al_{ji}=1$ for every $1\leq i, j\leq n$, and $\deg x_i = 1$ for every $1\leq i \leq n$.
\item We say that $S$ is a graded ($\pm 1$)-skew polynomial algebra if
\[S = k\ang{x_1, \dots, x_n}/(x_ix_j-\ep_{ij}x_jx_i)_{1\leq i, j\leq n}\]
is a graded skew polynomial algebra such that $\ep_{ij}$ equals either $1$ or $-1$ for every $1\leq i, j\leq n, i\neq j$.
\end{enumerate}
\end{dfn}
Clearly, a graded polynomial algebra $k[x_1, \dots, x_n]$ generated in degree 1 is an example of a graded ($\pm 1$)-skew polynomial algebra.
Consider the element
\[f= x_1^2 + x_2^2 + \cdots + x_n^2\]
of a graded skew polynomial algebra $S= k\ang{x_1, \dots, x_n}/(x_ix_j-\al_{ij}x_jx_i)$.
Then we notice that $f$ is normal if and only if $f$ is central if and only if $S$ is a ($\pm 1$)-skew polynomial algebra.

Let $S=k\ang{x_1, \dots, x_n}/(x_ix_j-\ep_{ij}x_jx_i)$ be a graded ($\pm 1$)-skew polynomial algebra so that
$f= x_1^2 + x_2^2 + \cdots + x_n^2 \in S$
is a homogeneous regular central element. Let $A$ be the graded quotient algebra $S/(f)$.
Since $S$ is a noetherian AS-regular algebra of dimension $n$ and $A$ is a noetherian AS-Gorenstein algebra of dimension $n-1$,
$A$ is regarded as a homogeneous coordinate ring of a quadric hypersurface in a ($\pm 1$)-skew projective space.
The main focus of this paper is to determine the structure of $\uCM(A)$ from a geometric data associated to $S$ called the point scheme of $S$.
Based on our experiments, we propose the following conjecture.

\begin{conj} \label{conj.main}
Let $S=k\ang{x_1, \dots, x_n}/(x_ix_j-\ep_{ij}x_jx_i)$ be a graded ($\pm 1$)-skew polynomial algebra, $f= x_1^2 + x_2^2 + \cdots + x_n^2 \in S$, and $A=S/(f)$.
Let $\ell$ be the number of irreducible components of the point scheme of $S$ which are isomorphic to $\P^1$.
\begin{enumerate}
\item If $n$ is odd, then 
\[
\binom{2m-1}{2} < \ell \leq \binom{2m+1}{2} \Longleftrightarrow \uCM(A) \cong \Db(\mod k^{2^{2m}})
\]
for $m \in \N$ where we consider $\binom{-1}{2}=-\infty, \binom{1}{2}=0$.
\item If $n$ is even, then
\[
\binom{2m}{2} < \ell \leq \binom{2m+2}{2} \Longleftrightarrow \uCM(A) \cong \Db(\mod k^{2^{2m+1}})
\]
for $m \in \N$ where we consider $\binom{0}{2}=-\infty$.
\end{enumerate}
\end{conj}

We prove the following result.

\begin{thm}[{\thmref{thm.main}}]
Let $S=k\ang{x_1, \dots, x_n}/(x_ix_j-\ep_{ij}x_jx_i)$ be a graded ($\pm 1$)-skew polynomial algebra, $f= x_1^2 + x_2^2 + \cdots + x_n^2 \in S$, and $A=S/(f)$.
Assume that $n\leq 5$.
Let $\ell$ be the number of irreducible components of the point scheme of $S$ which are isomorphic to $\P^1$.
\begin{enumerate}
\item If $n$ is odd, then $\ell\leq 10$ and
\begin{align*}
\ell =0 &\Longleftrightarrow \uCM(A) \cong \Db(\mod k),\\
0< \ell \leq 3 &\Longleftrightarrow \uCM(A) \cong \Db(\mod k^4),\\
3< \ell \leq 10 &\Longleftrightarrow \uCM(A) \cong \Db(\mod k^{16}).
\end{align*}
\item If $n$ is even, then $\ell\leq 6$ and
\begin{align*}
0\leq \ell \leq 1 &\Longleftrightarrow \uCM(A) \cong \Db(\mod k^2),\\
1< \ell \leq 6 &\Longleftrightarrow \uCM(A) \cong \Db(\mod k^8).
\end{align*}
\end{enumerate}
\end{thm}

This theorem asserts that \conjref{conj.main} is true if $n\leq 5$.

\section{Preliminaries}

\subsection{Notations}
For an algebra $A$, we denote by $\Mod A$ the category of right $A$-modules, and by $\mod A$ the full subcategory consisting of finitely generated modules.
The bounded derived category of $\mod A$ is denoted by $\Db(\mod A)$.

For a connected graded algebra $A$, that is, $A=\bigoplus_{i \in \N} A_i$ with $A_0 =k$,
we denote by $\GrMod A$ the category of graded right $A$-modules with $A$-module homomorphisms of degree zero,
and by $\grmod A$ the full subcategory consisting of finitely generated graded modules.

Let $A$ be a noetherian AS-Gorenstein algebra of dimension $n$ (see \cite[Section 1]{Jlc} for the definition).
We define the local cohomology modules of $M \in \grmod A$ by
$\H^i_\m(M):= \lim _{n \to \infty} \Ext^i_{A}(A/A_{\geq n}, M)$.
It is well-known that $\H_\m^i(A)=0$ for all $i\neq n$.
We say that $M \in \grmod A$ is graded maximal Cohen-Macaulay if $\H_{\m}^i(M)=0$ for all $i\neq n$.
We denote by $\CM (A)$ the full subcategory of $\grmod A$ consisting of graded maximal Cohen-Macaulay modules.

The stable category of graded maximal Cohen-Macaulay modules, denoted by $\uCM (A)$, has the same objects as $\CM (A)$
and the morphism set is given by
\[ \Hom_{\uCM (A)}(M, N) = \Hom_{\grmod A}(M,N)/P(M,N) \]
for any $M,N \in \CM (A)$, where $P(M,N)$ consists of degree zero $A$-module homomorphisms that factor
through a projective module in $\grmod A$.
Since $A$ is AS-Gorenstein, $\uCM (A)$ is a triangulated category with respect to the translation functor $M[-1]= \Om M$ (the syzygy of $M$)
by \cite[Theorem 3.1]{SV}.

\subsection{The algebra $C(A)$}
The method we use is due to Smith and Van den Bergh \cite{SV}; it was originally developed by Buchweitz, Eisenbud, and Herzog \cite{BEH}.

Let $S$ be an $n$-dimensional noetherian AS-regular algebra with the Hilbert series $H_S(t)=(1-t)^{-n}$.
Then $S$ is Koszul by \cite[Theorem 5.11]{Sm}.
Let $f \in S$ be a homogeneous regular central element of degree 2, and let $A =S/(f)$.
Then $A$ is Koszul by \cite[Lemma 5.1 (1)]{SV}, and there exists a central regular element $w \in A^!_2$ such that $A^!/(w) \cong S^!$ by \cite[Lemma 5.1 (2)]{SV}.
We can define the algebra
\[ C(A) := A^![w^{-1}]_0. \]
By \cite[Lemma 5.1 (3)]{SV}, we have $\dim_k C(A) = \dim_k (S^!)^{(2)} = 2^{n-1}$.

\begin{thm}[{\cite[Proposition 5.2]{SV}}] \label{thm.sv}
Let notation be as above. Then $\uCM(A) \cong \Db(\mod C(A))$.
\end{thm}

\subsection{The point schemes of skew polynomial algebras}

Let $S$ be a quantum polynomial algebra of dimension $n$ (see \cite[Definition 2.1]{Mcp} for the definition).

\begin{dfn}
A graded module $M \in \GrMod S$ is called a point module if $M$ is cyclic, generated in degree 0, and $H_M(t) = (1-t)^{-1}$.
\end{dfn}
If $M \in \GrMod S$ is a point module, then $M$ be written as a quotient $S/(g_1S+ g_2S+ \cdots + g_{n-1}S)$ with linearly independent $g_1, \dots, g_{n-1} \in S_1$
by \cite[Corollary 5.7, Theorem 3.8]{Mcp},
so we can associate it with a unique point $p_M := \cV(g_1, \dots, g_{n-1})$ in $\P({S_1^*}) = \P^{n-1}$.
Then the subset
\[ E:= \{ p_M \in \P^{n-1} \mid M \in \GrMod S \;\text{is a point module}\} \]
has a $k$-scheme structure by \cite{ATV}, and it is called the point scheme of $S$. 
Point schemes have a pivotal role in noncommutative algebraic geometry.

Thanks to the following result, we can compute the point scheme of a graded skew polynomial algebra.

\begin{thm}[{\cite[Proposition 4.2]{V}, \cite[Theorem 1 (1)]{BDL}}]\label{thm.ps}
Let $S = k\ang{x_1, \dots, x_n}/(x_ix_j-\al_{ij}x_jx_i)$ be a graded skew polynomial algebra.
Then the point scheme of $S$ is given by
\[ E= \bigcap_{\substack{1\leq i<j<k\leq n \\ \al_{ij}\al_{jk}\al_{ki} \neq 1}}\cV(x_ix_jx_k) \quad \subset \P^{n-1}. \]
\end{thm}

For $1\leq i_0,\dots,i_s\leq n$, we define the subspace
\[\P(i_1,\dots,i_s):= \bigcap_{\substack{1\leq j \leq n \\ j \neq i_1, \dots, j \neq i_s}} \cV(x_j) \; \subset \P^{n-1}.\]
It is easy to see that the point scheme of a graded skew polynomial algebra in 3 variables is isomorphic to 
$\P^2$ or $\P(2,3) \cup \P(1,3) \cup \P(1,2)$.
The following is the classification of the point schemes of graded skew polynomial algebras in 4 variables.

\begin{prop}[{\cite[Corollary 5.1]{V}, \cite[Section 4.2]{BDL}}] \label{prop.4ps}
Let $S = k\ang{x_1, x_2, x_3, x_4}/(x_ix_j-\al_{ij}x_jx_i)$ be a graded skew polynomial algebra in 4 variables.
Then the point scheme of $S$ is isomorphic one of the following:
\begin{itemize}
\item $\P^3$;
\item $\P(1,2,4) \cup \P(1,2,3) \cup \P(3,4)$;
\item $\P(2,3,4) \cup \P(1,4) \cup \P(1,3) \cup \P(1,2)$;
\item $\P(3,4) \cup \P(2,4) \cup \P(2,3) \cup \P(1,4) \cup \P(1,3) \cup \P(1,2)$.
\end{itemize}
\end{prop}

\section{Results}

Throughout this section,
\begin{itemize}
\item $S = k\ang{x_1, \dots, x_n}/(x_ix_j-\ep_{ij}x_jx_i)$ is a graded ($\pm 1$)-skew polynomial algebra,
\item $E$ is the point scheme of $S$,
\item $f= x_1^2 + x_2^2 + \cdots + x_n^2 \in S_2$ (a regular central element of $S$), and
\item $A=S/(f)$.
\end{itemize}
Note that $\ep_{ij} = \ep_{ji}$ holds for every $1\leq i, j\leq n$.

\begin{lem} \label{lem.c}
\begin{enumerate}
\item $A^!$ is isomorphic to $k\ang{x_1, \dots, x_n}/(\ep_{ij}x_ix_j+x_jx_i, x_n^2-x_i^2)_{1\leq i, j\leq n, i\neq j}$.
\item $w=x_n^2 \in A^!_2$ is a central regular element such that $A^!/(w) \cong S^!$.
\item $C(A) := A^![w^{-1}]_0$ is isomorphic to 
\[k\ang{t_1, \dots, t_{n-1}}/(t_it_j+\ep_{ni}\ep_{ij}\ep_{jn}t_jt_i,\, t_i^2-1)_{1\leq i, j\leq n-1, i\neq j}.\]
\end{enumerate}
\end{lem}

\begin{proof}
(1) and (2) follow from direct calculation.

(3) Since $S$ has a $k$-basis $\{x_1^{i_1}x_2^{i_2} \cdots x_n^{i_n} \mid i_1, i_2,\dots,i_n \geq 0 \}$, and 
\[ (x_nx_iw^{-1})(x_nx_jw^{-1}) = x_nx_ix_nx_jw^{-2} = -\ep_{ni}x_n^2x_ix_jw^{-2} = -\ep_{ni}x_ix_jw^{-1} \]
in $C(A)$ for $1\leq i\leq n-1, i\neq j$, it follows that $\{x_nx_1w^{-1}, \dots, x_nx_{n-1}w^{-1}\}$ is a set of generators of $C(A)$.
Put $t_i := x_nx_{i}w^{-1}$ for $1\leq i\leq n-1$.
Since
\begin{align*}
t_it_j &=  (x_nx_iw^{-1})(x_nx_jw^{-1}) = -\ep_{ni}x_ix_jw^{-1} = \ep_{ni}\ep_{ji}x_jx_iw^{-1} \\
&= -\ep_{ni}\ep_{ij}\ep_{jn}(-\ep_{nj}x_jx_iw^{-1}) = -\ep_{ni}\ep_{ij}\ep_{jn}(x_nx_jw^{-1})(x_nx_iw^{-1}) =  -\ep_{ni}\ep_{ij}\ep_{jn}t_jt_i,
\end{align*}
for $1\leq i, j\leq n-1, i\neq j$, and
\begin{align*}
t_i^2 &= (x_nx_iw^{-1})(x_nx_iw^{-1}) = -\ep_{ni}x_i^2w^{-1} = -\ep_{ni}x_n^2w^{-1} = -\ep_{ni}
\end{align*}
for $1\leq i\leq n-1$, we have a surjection
$k\ang{t_1, \dots, t_{n-1}}/(t_it_j+\ep_{ni}\ep_{ij}\ep_{jn}t_jt_i,\, t_i^2+\ep_{ni}) \to C(A)$.
This is an isomorphism because the algebras have the same dimension.
Since $\ep_{ni} \neq 0$ for $1\leq i\leq n-1$, the homomorphism defined by $t_i \to \sqrt{-\ep_{ni}}t_i$ induces the isomorphism
\[ k\ang{t_1, \dots, t_{n-1}}/(t_it_j+\ep_{ni}\ep_{ij}\ep_{jn}t_jt_i,\, t_i^2+\ep_{ni}) \xrightarrow{\sim} k\ang{t_1, \dots, t_{n-1}}/(t_it_j+\ep_{ni}\ep_{ij}\ep_{jn}t_jt_i,\, t_i^2-1). \]
\end{proof}

\begin{prop}\label{prop.pc}
\begin{enumerate}
\item If $E = \P^{n-1}$, then $C(A)$ is isomorphic to
\[C_{+}:=k\ang{t_1, \dots, t_{n-1}}/(t_it_j+t_jt_i,\, t_i^2-1)_{1\leq i, j\leq n-1, i\neq j}.\]
\item $E = \bigcup_{1\leq i< j\leq n} \P(i,j)$
if and only if $C(A)$ is isomorphic to
\[C_{-}:=k\ang{t_1, \dots, t_{n-1}}/(t_it_j-t_jt_i,\, t_i^2-1)_{1\leq i, j\leq n-1, i\neq j}.\]
\end{enumerate}
\end{prop}

\begin{proof}
First note that
\begin{align}\label{eq.3c}
\ep_{ij}\ep_{jk}\ep_{ki} = (\ep_{ni}\ep_{ij}\ep_{jn})(\ep_{nj}\ep_{jk}\ep_{kn})(\ep_{nk}\ep_{ki}\ep_{in})
\end{align}
for $1\leq i<j<k\leq n$.

(1) By \thmref{thm.ps}, (\ref{eq.3c}), and \lemref{lem.c} (3), it follows that
\begin{align*}
E = \P^{n-1} 
&\Longleftrightarrow \ep_{ij}\ep_{jk}\ep_{ki} = 1 \;\text{for every}\; 1\leq i<j<k\leq n\\
&\Longleftrightarrow \ep_{ni}\ep_{ij}\ep_{jn} = 1 \;\text{for every}\; 1\leq i<j\leq n\\
&\Longrightarrow C(A) \cong C_{+}.
\end{align*}

(2) By \thmref{thm.ps}, (\ref{eq.3c}), and \lemref{lem.c} (3), it follows that
\begin{align*}
E = \bigcup_{1\leq i< j\leq n} \P(i,j)
&\Longleftrightarrow \ep_{ij}\ep_{jk}\ep_{ki} \neq 1 \;\text{for every}\; 1\leq i<j<k\leq n\\
&\Longleftrightarrow \ep_{ij}\ep_{jk}\ep_{ki} = -1 \;\text{for every}\; 1\leq i<j<k\leq n\\
&\Longleftrightarrow \ep_{ni}\ep_{ij}\ep_{jn} = -1 \;\text{for every}\; 1\leq i<j\leq n\\
&\Longleftrightarrow C(A) \cong C_{-}.
\end{align*}
Here the last $\Longleftarrow$ is by commutativity of $C(A)$.
\end{proof}

\begin{thm}\label{thm.gst}
\begin{enumerate}
\item If $E=\P^{n-1}$ and $n$ is odd, then $\uCM(S/(f)) \cong \Db(\mod k)$.
\item If $E=\P^{n-1}$ and $n$ is even, then $\uCM(S/(f)) \cong \Db(\mod k^2)$.
\item $E = \bigcup_{1\leq i< j\leq n} \P(i,j)$ if and only if $\uCM(S/(f)) \cong \Db(\mod k^{2^{n-1}})$.
\end{enumerate}
\end{thm}

\begin{proof}
Since $C_{+}$ is a Clifford algebra over $k$, it is known that
\begin{align} \label{eq.ci}
C_{+} \cong \begin{cases} M_{2^{(n-1)/2}}(k) &\text{if $n$ is odd},\\ M_{2^{(n-2)/2}}(k)^2 &\text{if $n$ is even}, \end{cases}
\end{align}
so
\[\mod C_{+} \cong \begin{cases}  \mod M_{2^{(n-1)/2}}(k) \cong \mod k &\text{if $n$ is odd},\\ \mod M_{2^{(n-2)/2}}(k)^2 \cong \mod k^2 &\text{if $n$ is even}. \end{cases} \]
Thus (1) and (2) follow from \thmref{thm.sv} and \propref{prop.pc} (1).

We next show (3). If $E = \bigcup_{1\leq i< j\leq n} \P(i,j)$, then $C(A) \cong C_{-}$ by \propref{prop.pc} (2).
Since $C_{-}$ is isomorphic to the group algebra of $(\Z_2)^{n-1}$ over $k$, we have $C_{-} \cong k^{2^{n-1}}$, so it follows that $\uCM(S/(f)) \cong \Db(\mod k^{2^{n-1}})$
by \thmref{thm.sv}.
Conversely, if $\uCM(S/(f)) \cong \Db(\mod k^{2^{n-1}})$, then $\Db(\mod C(A)) \cong \Db(\mod k^{2^{n-1}})$ by \thmref{thm.sv}.
Since $\dim_kC(A)=2^{n-1}$, it follows that $C(A) \cong k^{2^{n-1}} \cong C_{-}$.
Hence $E = \bigcup_{1\leq i< j\leq n} \P(i,j)$ by \propref{prop.pc} (2).
\end{proof}

Note that \thmref{thm.gst} (1), (2) recover \thmref{thm.Kn}, and \thmref{thm.gst} (3) shows that a new phenomenon appears in the noncommutative case.
We can now give an explicit classification of $\uCM(A)$ in the case $n\leq 3$
(the case $n=1$ is clear; see \thmref{thm.Kn} (1)).

\begin{cor}\label{cor.gst}
\begin{enumerate}
\item If $n=2$, then $E=\P^1$ and $\uCM(A) \cong \Db(\mod k^2)$.
\item If $n=3$, then
\[
\renewcommand{\arraystretch}{1.2}
\begin{array}{ll}
E=\P^2 &\Longleftrightarrow \uCM(A) \cong \Db(\mod k), \\
E=\P(2,3) \cup \P(1,3) \cup \P(1,2) &\Longleftrightarrow \uCM(A) \cong \Db(\mod k^4).
\end{array}
\]
\end{enumerate}
\end{cor}

\begin{proof}
These follow from \thmref{thm.ps} and \thmref{thm.gst}.
\end{proof}

As we will see later, the converse of \thmref{thm.gst} (1), (2) does not hold in general.
So, in order to give a classification for the cases $n=4$ and $n=5$, we need a precise computation.

For a permutation $\si \in {\mathfrak S}_n$, we have an isomorphism
\[ S=k\ang{x_1, \dots, x_n}/(x_ix_j-\ep_{ij}x_jx_i) \xrightarrow[\vph]{\sim} k\ang{x_1, \dots, x_n}/(x_{\si(i)}x_{\si(j)}-\ep_{ij}x_{\si(j)}x_{\si(i)}) =:  S_\si\]
between graded ($\pm 1$)-skew polynomial algebras, which we call a permutation isomorphism.
Since $\vph$ preserves $f$, it induces an isomorphism
\begin{align*} \label{eq.pi}
A = S/(f) \xrightarrow{\sim} S_\si /(f),
\end{align*}
which we also call a permutation isomorphism.

\begin{lem} \label{lem.clp1}
If $n=4$, then, via a permutation isomorphism, $S$ is isomorphic to a graded ($\pm 1$)-skew polynomial algebra whose point scheme is one of the following:
\begin{enumerate}
\item[(4a)] $\P^3$;
\item[(4b)] $\P(1,2,4) \cup \P(1,2,3) \cup \P(3,4)$;
\item[(4c)] $\P(3,4) \cup \P(2,4) \cup \P(2,3) \cup \P(1,4) \cup \P(1,3) \cup \P(1,2)$.
\end{enumerate}
\end{lem}

\begin{proof}
First, via a permutation isomorphism, $S$ is isomorphic to one of the following:
\begin{enumerate}
\item[(4\rnum 1)] a graded ($\pm 1$)-skew polynomial algebra with
$\ep_{41}\ep_{12}\ep_{24}=\ep_{41}\ep_{13}\ep_{34}=\ep_{42}\ep_{23}\ep_{34}=1$;
\item[(4\rnum 2)] a graded ($\pm 1$)-skew polynomial algebra with
$\ep_{41}\ep_{12}\ep_{24}=\ep_{41}\ep_{13}\ep_{34}=1,\; \ep_{42}\ep_{23}\ep_{34}=-1$;
\item[(4\rnum 3)] a graded ($\pm 1$)-skew polynomial algebra with
$\ep_{41}\ep_{12}\ep_{24}=1,\; \ep_{41}\ep_{13}\ep_{34}=\ep_{42}\ep_{23}\ep_{34}=-1$;
\item[(4\rnum 4)] a graded ($\pm 1$)-skew polynomial algebra with
$\ep_{41}\ep_{12}\ep_{24}=\ep_{41}\ep_{13}\ep_{34}=\ep_{42}\ep_{23}\ep_{34}=-1$.
\end{enumerate}
Note that the above follows from (\ref{eq.3c}) and the classification of simple graphs of order $3$:
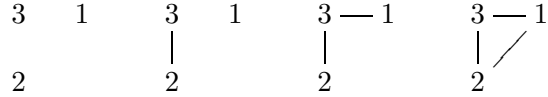
\begin{align*}
\xymatrix@C=1pc@R=1pc{
3 &1 \\2
}
\qquad
\xymatrix@C=1pc@R=1pc{
3 \ar@{-}[d] &1 \\2
}
\qquad
\xymatrix@C=1pc@R=1pc{
3 \ar@{-}[d] \ar@{-}[r] &1 \\2
}
\qquad
\xymatrix@C=1pc@R=1pc{
3 \ar@{-}[d] \ar@{-}[r] &1 \ar@{-}[ld] \\2
}
\end{align*}
(we define $\ep_{4i}\ep_{ij}\ep_{j4}=-1$ if $\{i,j\}$ is an edge in the graph, and $\ep_{4i}\ep_{ij}\ep_{j4}=1$ otherwise).

The point scheme of an algebra in the case (4\rnum 1) is $\P^3$, so this is (4a).

The point scheme of an algebra in the case (4\rnum 3) is 
$$\cV(x_1x_3x_4) \cap \cV(x_2x_3x_4) = \cV(x_3) \cup \cV(x_4) \cup \cV(x_1, x_2),$$
so this is (4b).
The point scheme of an algebra in the case (4\rnum 2) is $\cV(x_1x_2x_3) \cap \cV(x_2x_3x_4) = \cV(x_2) \cup \cV(x_3) \cup \cV(x_1, x_4)$, so
an algebra in the case (4\rnum 2) is isomorphic to an algebra in the case (4\rnum 3) via the permutation isomorphism induced by
$\si = \left(\begin{smallmatrix} 1 &2 &3 &4 \\ 1 &4 &3 &2 \end{smallmatrix}\right)$.

The point scheme of an algebra in the case (4\rnum 4) is $\bigcap_{1\leq i<j<k\leq 4}\cV(x_ix_jx_k) = \bigcup_{1\leq i<j\leq 4} \cV(x_i, x_j)$, so this is (4c).
\end{proof}

\begin{rem}
It follows from \lemref{lem.clp1} that not every point scheme in \propref{prop.4ps} appears as the point scheme of a graded ($\pm 1$)-skew polynomial algebra.
\end{rem}

\begin{lem} \label{lem.clp2}
If $n=5$, then, via a permutation isomorphism, $S$ is isomorphic to a graded ($\pm 1$)-skew polynomial algebra whose point scheme is one of the following:
\begin{enumerate}
\item[(5a)] $\P^4$;
\item[(5b)] $\P(1,2,3,5) \cup \P(1,2,3,4) \cup \P(4,5)$;
\item[(5c)] $\P(1,2,3,4) \cup \P(3,4,5) \cup \P(1,2,5)$;
\item[(5d)] $\P(3,4,5) \cup \P(1,4,5) \cup \P(1,2,5) \cup \P(1,2,3) \cup \P(2,3,4)$;
\item[(5e)] $\P(1,3,5) \cup \P(1,3,4) \cup \P(1,2,5) \cup \P(1,2,4) \cup \P(4,5) \cup \P(2,3)$;
\item[(5f)] $\P(1,2,5) \cup \P(1,2,4) \cup \P(1,2,3) \cup \P(4,5) \cup \P(3,5) \cup \P(3,4)$;
\item[(5g)] $\P(4,5) \cup \P(3,5) \cup \P(3,4) \cup \P(2,5) \cup \P(2,4) \cup \P(2,3)\cup \P(1,5) \cup \P(1,4) \cup \P(1,3) \cup \P(1,2)$.
\end{enumerate}
\end{lem}

\begin{proof}
First, via a permutation isomorphism, $S$ is isomorphic to one of the following:
\begin{enumerate}
\item[(5\rnum 1)] a graded ($\pm 1$)-skew polynomial algebra with
\[\ep_{51}\ep_{12}\ep_{25}=1,\; \ep_{51}\ep_{13}\ep_{35}=1,\; \ep_{51}\ep_{14}\ep_{45}=1,\;
  \ep_{52}\ep_{23}\ep_{35}=1,\; \ep_{52}\ep_{24}\ep_{45}=1,\; \ep_{53}\ep_{34}\ep_{45}=1;  \]
\item[(5\rnum 2)] a graded ($\pm 1$)-skew polynomial algebra with
\[\ep_{51}\ep_{12}\ep_{25}=1,\; \ep_{51}\ep_{13}\ep_{35}=1,\; \ep_{51}\ep_{14}\ep_{45}=1,\;
  \ep_{52}\ep_{23}\ep_{35}=1,\; \ep_{52}\ep_{24}\ep_{45}=1,\; \ep_{53}\ep_{34}\ep_{45}=-1;  \]
\item[(5\rnum 3)] a graded ($\pm 1$)-skew polynomial algebra with
\[\ep_{51}\ep_{12}\ep_{25}=1,\; \ep_{51}\ep_{13}\ep_{35}=1,\; \ep_{51}\ep_{14}\ep_{45}=1,\;
  \ep_{52}\ep_{23}\ep_{35}=1,\; \ep_{52}\ep_{24}\ep_{45}=-1,\; \ep_{53}\ep_{34}\ep_{45}=-1;  \]
\item[(5\rnum 4)] a graded ($\pm 1$)-skew polynomial algebra with
\[\ep_{51}\ep_{12}\ep_{25}=-1,\; \ep_{51}\ep_{13}\ep_{35}=1,\; \ep_{51}\ep_{14}\ep_{45}=1,\;
  \ep_{52}\ep_{23}\ep_{35}=1,\; \ep_{52}\ep_{24}\ep_{45}=1,\; \ep_{53}\ep_{34}\ep_{45}=-1;  \]
\item[(5\rnum 5)] a graded ($\pm 1$)-skew polynomial algebra with
\[\ep_{51}\ep_{12}\ep_{25}=1,\; \ep_{51}\ep_{13}\ep_{35}=1,\; \ep_{51}\ep_{14}\ep_{45}=-1,\;
  \ep_{52}\ep_{23}\ep_{35}=1,\; \ep_{52}\ep_{24}\ep_{45}=-1,\; \ep_{53}\ep_{34}\ep_{45}=-1;  \]
\item[(5\rnum 6)] a graded ($\pm 1$)-skew polynomial algebra with
\[\ep_{51}\ep_{12}\ep_{25}=1,\; \ep_{51}\ep_{13}\ep_{35}=-1,\; \ep_{51}\ep_{14}\ep_{45}=1,\;
  \ep_{52}\ep_{23}\ep_{35}=-1,\; \ep_{52}\ep_{24}\ep_{45}=-1,\; \ep_{53}\ep_{34}\ep_{45}=1;  \]
\item[(5\rnum 7)] a graded ($\pm 1$)-skew polynomial algebra with
\[\ep_{51}\ep_{12}\ep_{25}=1,\; \ep_{51}\ep_{13}\ep_{35}=-1,\; \ep_{51}\ep_{14}\ep_{45}=-1,\;
  \ep_{52}\ep_{23}\ep_{35}=1,\; \ep_{52}\ep_{24}\ep_{45}=1,\; \ep_{53}\ep_{34}\ep_{45}=-1;  \]
\item[(5\rnum 8)] a graded ($\pm 1$)-skew polynomial algebra with
\[\ep_{51}\ep_{12}\ep_{25}=1,\; \ep_{51}\ep_{13}\ep_{35}=-1,\; \ep_{51}\ep_{14}\ep_{45}=-1,\;
  \ep_{52}\ep_{23}\ep_{35}=-1,\; \ep_{52}\ep_{24}\ep_{45}=-1,\; \ep_{53}\ep_{34}\ep_{45}=1;  \]
\item[(5\rnum 9)] a graded ($\pm 1$)-skew polynomial algebra with
\[\ep_{51}\ep_{12}\ep_{25}=1,\; \ep_{51}\ep_{13}\ep_{35}=1,\; \ep_{51}\ep_{14}\ep_{45}=-1,\;
  \ep_{52}\ep_{23}\ep_{35}=-1,\; \ep_{52}\ep_{24}\ep_{45}=-1,\; \ep_{53}\ep_{34}\ep_{45}=-1;  \]
\item[(5\rnum {10})] a graded ($\pm 1$)-skew polynomial algebra with
\[\ep_{51}\ep_{12}\ep_{25}=1,\; \ep_{51}\ep_{13}\ep_{35}=-1,\; \ep_{51}\ep_{14}\ep_{45}=-1,\;
  \ep_{52}\ep_{23}\ep_{35}=-1,\; \ep_{52}\ep_{24}\ep_{45}=-1,\; \ep_{53}\ep_{34}\ep_{45}=-1;  \]
\item[(5\rnum {11})] a graded ($\pm 1$)-skew polynomial algebra with
\[\ep_{51}\ep_{12}\ep_{25}=-1,\; \ep_{51}\ep_{13}\ep_{35}=-1,\; \ep_{51}\ep_{14}\ep_{45}=-1,\;
  \ep_{52}\ep_{23}\ep_{35}=-1,\; \ep_{52}\ep_{24}\ep_{45}=-1,\; \ep_{53}\ep_{34}\ep_{45}=-1;  \]
\end{enumerate}
Note that the above follows from (\ref{eq.3c}) and the classification of simple graphs of order $4$:
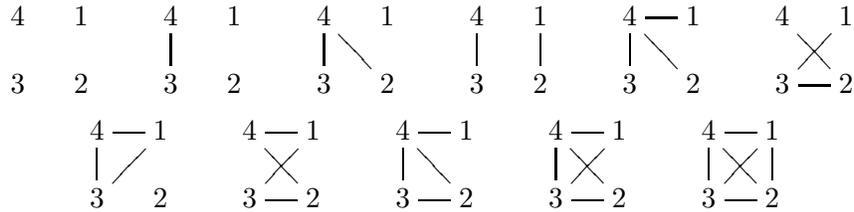
\begin{align*}
\xymatrix@C=1pc@R=1pc{
4 &1 \\3 &2
}
\qquad
\xymatrix@C=1pc@R=1pc{
4 \ar@{-}[d] &1 \\3 &2
}
\qquad
\xymatrix@C=1pc@R=1pc{
4 \ar@{-}[d] \ar@{-}[rd] &1 \\3 &2
}
\qquad
\xymatrix@C=1pc@R=1pc{
4 \ar@{-}[d] &1\ar@{-}[d]  \\3  &2
}
\qquad
\xymatrix@C=1pc@R=1pc{
4 \ar@{-}[d] \ar@{-}[rd] \ar@{-}[r] &1 \\3 &2
}
\qquad
\xymatrix@C=1pc@R=1pc{
4 \ar@{-}[rd] &1\ar@{-}[ld] \\3 \ar@{-}[r]&2
}
\\
\xymatrix@C=1pc@R=1pc{
4 \ar@{-}[d] \ar@{-}[r] &1 \ar@{-}[ld] \\3 &2
}
\qquad
\xymatrix@C=1pc@R=1pc{
4 \ar@{-}[rd] \ar@{-}[r] &1\ar@{-}[ld] \\3 \ar@{-}[r]&2
}
\qquad
\xymatrix@C=1pc@R=1pc{
4 \ar@{-}[rd] \ar@{-}[r] \ar@{-}[d]&1 \\3 \ar@{-}[r]&2
}
\qquad
\xymatrix@C=1pc@R=1pc{
4 \ar@{-}[rd] \ar@{-}[r] \ar@{-}[d]&1 \ar@{-}[ld] \\3 \ar@{-}[r]&2
}
\qquad
\xymatrix@C=1pc@R=1pc{
4 \ar@{-}[rd] \ar@{-}[r] \ar@{-}[d]&1 \ar@{-}[ld] \ar@{-}[d] \\3 \ar@{-}[r]&2
}
\qquad\;\;
\end{align*}
(we define $\ep_{5i}\ep_{ij}\ep_{j5}=-1$ if $\{i,j\}$ is an edge in the graph, and $\ep_{5i}\ep_{ij}\ep_{j5}=1$ otherwise).

The point scheme of an algebra in the case (5\rnum 1) is $\P^4$, so this is (5a).

The point scheme of an algebra in the case (5\rnum 5) is 
$$\cV(x_1x_4x_5) \cap \cV(x_2x_4x_5) \cap \cV(x_3x_4x_5) = \cV(x_4) \cup \cV(x_5) \cup \cV(x_1, x_2, x_3),$$
so this is (5b).
The point scheme of an algebra in the case (5\rnum 2) is $\cV(x_1x_3x_4) \cap \cV(x_2x_3x_4) \cap \cV(x_3x_4x_5) = \cV(x_3) \cup \cV(x_4) \cup \cV(x_1, x_2, x_5)$, so
an algebra in the case (5\rnum 2) is isomorphic to an algebra in the case (5\rnum 5) via the permutation isomorphism induced by
$\si = \left(\begin{smallmatrix} 1 &2 &3 &4 &5 \\ 1 &2 &5 &4 &3 \end{smallmatrix}\right)$.

The point scheme of an algebra in the case (5\rnum 8) is
$$\cV(x_1x_3x_5) \cap \cV(x_1x_4x_5) \cap \cV(x_2x_3x_5) \cap \cV(x_2x_4x_5) = \cV(x_5) \cup \cV(x_1, x_2) \cup \cV(x_3, x_4),$$
so this is (5c).
The point scheme of an algebra in the case (5\rnum 3) is $\cV(x_1x_2x_4) \cap \cV(x_1x_3x_4) \cap \cV(x_2x_4x_5) \cap \cV(x_3x_4x_5) = \cV(x_4) \cup \cV(x_1, x_5) \cup \cV(x_2, x_3)$, so
an algebra in the case (5\rnum 3) is isomorphic to an algebra in the case (5\rnum 8) via the permutation isomorphism induced by
$\si = \left(\begin{smallmatrix} 1 &2 &3 &4 &5 \\ 1 &4 &3 &5 &2 \end{smallmatrix}\right)$.

The point scheme of an algebra in the case (5\rnum 6) is
\begin{align*}
&\cV(x_1x_2x_4) \cap \cV(x_1x_3x_4) \cap \cV(x_1x_3x_5) \cap \cV(x_2x_3x_5) \cap \cV(x_2x_4x_5)\\
&= \cV(x_1, x_2) \cup \cV(x_2, x_3) \cup \cV(x_3, x_4)\cup \cV(x_4, x_5) \cup \cV(x_5, x_1),
\end{align*}
so this is (5d).

The point scheme of an algebra in the case (5\rnum 9) is
\begin{align*}
&\cV(x_1x_2x_3) \cap \cV(x_1x_4x_5) \cap \cV(x_2x_3x_4) \cap \cV(x_2x_3x_5) \cap \cV(x_2x_4x_5) \cap \cV(x_3x_4x_5)\\
&= \cV(x_2, x_4) \cup \cV(x_2, x_5) \cup \cV(x_3, x_4)\cup \cV(x_3, x_5) \cup \cV(x_1, x_2, x_3) \cup \cV(x_1, x_4, x_5),
\end{align*}
so this is (5e).
The point scheme of an algebra in the case (5\rnum 4) is $\cV(x_1x_2x_3) \cap \cV(x_1x_2x_4) \cap \cV(x_1x_2x_5) \cap \cV(x_1x_3x_4) \cap \cV(x_2x_3x_4) \cap \cV(x_3x_4x_5)
 = \cV(x_1, x_3) \cup \cV(x_1, x_4) \cup \cV(x_2, x_3)\cup \cV(x_2, x_4) \cup \cV(x_1, x_2, x_5) \cup \cV(x_3, x_4, x_5)$, so
an algebra in the case (5\rnum 4) is isomorphic to an algebra in the case (5\rnum 9) via the permutation isomorphism induced by
$\si = \left(\begin{smallmatrix} 1 &2 &3 &4 &5 \\ 3 &2 &5 &4 &1 \end{smallmatrix}\right)$.

The point scheme of an algebra in the case (5\rnum{10}) is
\begin{align*}
&\cV(x_1x_3x_4) \cap \cV(x_1x_3x_5) \cap \cV(x_1x_4x_5) \cap \cV(x_2x_3x_4) \cap \cV(x_2x_3x_5) \cap \cV(x_2x_4x_5) \cap \cV(x_3x_4x_5)\\
&= \cV(x_3, x_4) \cup \cV(x_3, x_5) \cup \cV(x_4, x_5)\cup \cV(x_1, x_2, x_3) \cup \cV(x_1, x_2, x_4) \cup \cV(x_1, x_2, x_5),
\end{align*}
so this is (5f).
The point scheme of an algebra in the case (5\rnum 7) is $\cV(x_1x_2x_3) \cap \cV(x_1x_2x_4) \cap \cV(x_1x_3x_4) \cap \cV(x_1x_3x_5) \cap \cV(x_1x_4x_5) \cap \cV(x_2x_3x_4) \cap \cV(x_3x_4x_5)
 = \cV(x_1, x_3) \cup \cV(x_1, x_4) \cup \cV(x_3, x_4)\cup \cV(x_1, x_2, x_5) \cup \cV(x_2, x_3, x_5) \cup \cV(x_2, x_4, x_5)$, so
an algebra in the case (5\rnum 7) is isomorphic to an algebra in the case (5\rnum{10}) via the permutation isomorphism induced by
$\si = \left(\begin{smallmatrix} 1 &2 &3 &4 &5 \\ 5 &2 &3 &4 &1 \end{smallmatrix}\right)$.

The point scheme of an algebra in the case (5\rnum {11}) is $\bigcap_{1\leq i<j<k\leq 5}\cV(x_ix_jx_k) = \bigcup_{1\leq i<j<k\leq 5} \cV(x_i, x_j, x_k)$, so this is (5g).
\end{proof}

To describe the algebras $C(A)$ appearing in \lemref{lem.c}, we show that the following algebras are isomorphic to algebras of the form $M_{i}(k)^{j}$.

\begin{lem} \label{lem.clc}
\begin{enumerate}
\item $C_{\rm \rnum 1}:= k\ang{t_1, t_2, t_3}/(t_1t_2+t_2t_1, t_1t_3+t_3t_1, t_2t_3-t_3t_2, t_1^2-1, t_2^2-1, t_3^2-1)$ is isomorphic to $M_2(k)^2$.
\item $C_{\rm \rnum 2}:= k\ang{t_1, t_2, t_3}/(t_1t_2+t_2t_1, t_1t_3-t_3t_1, t_2t_3-t_3t_2, t_1^2-1, t_2^2-1, t_3^2-1)$ is isomorphic to $M_2(k)^2$.
\item $C_{\rm \rnum 3}:= k\ang{t_1, t_2, t_3, t_4}/(t_1t_2+t_2t_1, t_1t_3+t_3t_1, t_1t_4-t_4t_1, t_2t_3+t_3t_2, t_2t_4-t_4t_2, t_3t_4-t_4t_3, t_1^2-1, t_2^2-1, t_3^2-1,t_4^2-1)$ is isomorphic to $M_2(k)^4$.
\item $C_{\rm \rnum 4}:= k\ang{t_1, t_2, t_3, t_4}/(t_1t_2+t_2t_1, t_1t_3-t_3t_1, t_1t_4-t_4t_1, t_2t_3-t_3t_2, t_2t_4-t_4t_2, t_3t_4+t_4t_3, t_1^2-1, t_2^2-1, t_3^2-1,t_4^2-1)$ is isomorphic to $M_4(k)$.
\item $C_{\rm \rnum 5}:= k\ang{t_1, t_2, t_3, t_4}/(t_1t_2+t_2t_1, t_1t_3-t_3t_1, t_1t_4+t_4t_1, t_2t_3-t_3t_2, t_2t_4-t_4t_2, t_3t_4+t_4t_3, t_1^2-1, t_2^2-1, t_3^2-1,t_4^2-1)$ is isomorphic to $M_4(k)$.
\item $C_{\rm \rnum 6}:= k\ang{t_1, t_2, t_3, t_4}/(t_1t_2+t_2t_1, t_1t_3+t_3t_1, t_1t_4-t_4t_1, t_2t_3-t_3t_2, t_2t_4-t_4t_2, t_3t_4-t_4t_3, t_1^2-1, t_2^2-1, t_3^2-1,t_4^2-1)$ is isomorphic to $M_2(k)^4$.
\item $C_{\rm \rnum 7}:= k\ang{t_1, t_2, t_3, t_4}/(t_1t_2+t_2t_1, t_1t_3-t_3t_1, t_1t_4-t_4t_1, t_2t_3-t_3t_2, t_2t_4-t_4t_2, t_3t_4-t_4t_3, t_1^2-1, t_2^2-1, t_3^2-1,t_4^2-1)$ is isomorphic to $M_2(k)^4$.
\end{enumerate}
\end{lem}

\begin{proof}
(1) Let
\begin{align*}
e_1 = \frac{1}{4}(1+t_2+t_3+t_2t_3),\quad
e_2 = \frac{1}{4}(1-t_2+t_3-t_2t_3),\\
e_3 = \frac{1}{4}(1+t_2-t_3-t_2t_3),\quad
e_4 = \frac{1}{4}(1-t_2-t_3+t_2t_3).
\end{align*}
Then they form a complete set of orthogonal idempotents of $C_{\rm \rnum 1}$. Since
\begin{align*}
&e_1t_1 = \frac{1}{4}(1+t_2+t_3+t_2t_3)t_1=\frac{1}{4}t_2(1-t_1-t_3+t_2t_3)=t_1e_4,\\
&e_2t_1 = \frac{1}{4}(1-t_2+t_3-t_2t_3)t_1=\frac{1}{4}t_2(1+t_1-t_3-t_2t_3)=t_1e_3,\\
&e_3t_1 = \frac{1}{4}(1+t_2-t_3-t_2t_3)t_1=\frac{1}{4}t_2(1-t_1+t_3-t_2t_3)=t_1e_2,\\
&e_4t_1 = \frac{1}{4}(1-t_2-t_3+t_2t_3)t_1=\frac{1}{4}t_2(1+t_1+t_3+t_2t_3)=t_1e_1,
\end{align*}
it follows that the map $M_2(k)^2 \to C_{\rm \rnum 1};$
\[ \left(\begin{pmatrix} a_{11} &a_{12} \\ a_{21} &a_{22} \end{pmatrix}, \begin{pmatrix} b_{11} &b_{12} \\ b_{21} &b_{22} \end{pmatrix}\right) \mapsto 
\begin{matrix*}[l] \;\;\; a_{11}e_1&+a_{12}e_1t_1e_4 &+b_{11}e_2 &+b_{12}e_2t_1e_3 \\
+a_{21}e_4t_1e_1 &+a_{22}e_4 &+b_{21}e_3t_1e_2 &+ b_{22}e_3\end{matrix*}\]
is an isomorphism of algebras.

(2) Since $t_3$ commutes with $t_1, t_2$ in $C_{\rm \rnum 2}$, we have 
\[C_{\rm \rnum 2} \cong k\ang{t_1, t_2}/(t_1t_2+t_2t_1, t_1^2-1, t_2^2-1) \otimes_k k[t_3]/(t_3^2-1) \cong M_2(k) \otimes_k k^2 \cong M_2(k)^2\]
by (\ref{eq.ci}).

(3) Since $t_4$ commutes with $t_1, t_2, t_3$ in $C_{\rm \rnum 3}$, we have 
\begin{align*}
C_{\rm \rnum 3} &\cong k\ang{t_1, t_2, t_3}/(t_1t_2+t_2t_1, t_1t_3+t_3t_1, t_2t_3+t_3t_2, t_1^2-1, t_2^2-1, t_3^2-1) \otimes_k k[t_4]/(t_4^2-1)\\
&\cong M_2(k)^2 \otimes_k k^2 \cong M_2(k)^4
\end{align*}
by (\ref{eq.ci}).

(4) Since $t_3, t_4$ commute with $t_1, t_2$ in $C_{\rm \rnum 4}$, we have
\begin{align*}
C_{\rm \rnum 4} &\cong k\ang{t_1, t_2}/(t_1t_2+t_2t_1, t_1^2-1, t_2^2-1) \otimes_k k\ang{t_3, t_4}/(t_3t_4+t_4t_3, t_3^2-1, t_4^2-1)\\
&\cong M_2(k) \otimes_k M_2(k) \cong M_4(k)
\end{align*}
by (\ref{eq.ci}).

(5) Let
\begin{align*}
e_1 = \frac{1}{4}(1+t_1+t_3+t_1t_3),\quad
e_2 = \frac{1}{4}(1-t_1+t_3-t_1t_3),\\
e_3 = \frac{1}{4}(1+t_1-t_3-t_1t_3),\quad
e_4 = \frac{1}{4}(1-t_1-t_3+t_1t_3).
\end{align*}
Then they form a complete set of orthogonal idempotents of $C_{\rm \rnum 5}$. Similar to the proof of (1), we have
\begin{align*}
&e_1t_4 =t_4e_4, \qquad e_1t_2=t_2e_2, \qquad e_1t_4t_2 =t_4t_2e_3,\\
&e_2t_4 =t_4e_3, \qquad e_2t_2=t_2e_1, \qquad e_2t_4t_2 =t_4t_2e_4,\\
&e_3t_4 =t_4e_2, \qquad e_3t_2=t_2e_4, \qquad e_3t_4t_2 =t_4t_2e_1,\\
&e_4t_4 =t_4e_1, \qquad e_4t_2=t_2e_3, \qquad e_4t_4t_2 =t_4t_2e_2,
\end{align*}
so it follows that the map $M_4(k) \to C_{\rm \rnum 5};$
\[ \begin{pmatrix} a_{11} &a_{12} &a_{13} &a_{14} \\ a_{21} &a_{22} &a_{23} &a_{24} \\ a_{31} &a_{32} &a_{33} &a_{34} \\ a_{41} &a_{42} &a_{43} &a_{44} \end{pmatrix} \mapsto 
\begin{matrix*}[l]
\;\;\; a_{11}e_1&+a_{12}e_1t_4e_4 &+a_{13}e_1t_2e_2 &+a_{14}e_1t_4t_2e_3 \\
+a_{21}e_4t_4e_1 &+a_{22}e_4      &+a_{23}e_4t_4t_2e_2 &+a_{24}e_4t_2e_3 \\
+a_{31}e_2t_2e_1 &+a_{32}e_2t_4t_2e_4  &+a_{33}e_2&+a_{34}e_2t_4e_3\\
+a_{41}e_3t_4t_2e_1 &+a_{42}e_3t_2e_4  &+a_{43}e_3t_4e_2 &+a_{44}e_3
\end{matrix*}\]
is an isomorphism of algebras.

(6) Since $t_4$ commutes with $t_1, t_2, t_3$ in $C_{\rm \rnum 6}$, we have 
\begin{align*}
C_{\rm \rnum 6} &\cong k\ang{t_1, t_2, t_3}/(t_1t_2+t_2t_1, t_1t_3+t_3t_1, t_2t_3-t_3t_2, t_1^2-1, t_2^2-1, t_3^2-1) \otimes_k k[t_4]/(t_4^2-1)\\
&\cong M_2(k)^2 \otimes_k k^2 \cong M_2(k)^4
\end{align*}
by (1).

(7) Since $t_4$ commutes with $t_1, t_2, t_3$ in $C_{\rm \rnum 7}$, we have 
\begin{align*}
C_{\rm \rnum 7} &\cong k\ang{t_1, t_2, t_3}/(t_1t_2+t_2t_1, t_1t_3-t_3t_1, t_2t_3-t_3t_2, t_1^2-1, t_2^2-1, t_3^2-1) \otimes_k k[t_4]/(t_4^2-1)\\
&\cong M_2(k)^2 \otimes_k k^2 \cong M_2(k)^4
\end{align*}
by (2).
\end{proof}

\begin{thm} \label{thm.clm}
\begin{enumerate}
\item If $n=4$, then
\[
\renewcommand{\arraystretch}{1.2}
\begin{array}{ll}
 E\cong \P^3 \; \text{or} \; \P(1,2,4) \cup \P(1,2,3) \cup \P(3,4) &\Longleftrightarrow\uCM(A) \cong \Db(\mod k^2), \\
E=\P(3,4) \cup \P(2,4) \cup \P(2,3) \cup \P(1,4) \cup \P(1,3) \cup \P(1,2) &\Longleftrightarrow \uCM(A) \cong \Db(\mod k^8).
\end{array}
\]
\item If $n=5$, then
\[
\renewcommand{\arraystretch}{1.2}
\begin{array}{ll}
E \cong (5a), (5c), \;\text{or}\;(5d) &\Longleftrightarrow \uCM(A) \cong \Db(\mod k),\\
E \cong (5b), (5e), \;\text{or}\;(5f) &\Longleftrightarrow \uCM(A) \cong \Db(\mod k^4),\\
E=(5g) &\Longleftrightarrow \uCM(A) \cong \Db(\mod k^{16}),
\end{array}
\]
where
\begin{enumerate}
\item[(5a)] $\P^4$
\item[(5b)] $\P(1,2,3,5) \cup \P(1,2,3,4) \cup \P(4,5)$
\item[(5c)] $\P(1,2,3,4) \cup \P(3,4,5) \cup \P(1,2,5)$
\item[(5d)] $\P(3,4,5) \cup \P(1,4,5) \cup \P(1,2,5) \cup \P(1,2,3) \cup \P(2,3,4)$
\item[(5e)] $\P(1,3,5) \cup \P(1,3,4) \cup \P(1,2,5) \cup \P(1,2,4) \cup \P(4,5) \cup \P(2,3)$
\item[(5f)] $\P(1,2,5) \cup \P(1,2,4) \cup \P(1,2,3) \cup \P(4,5) \cup \P(3,5) \cup \P(3,4)$
\item[(5g)] $\P(4,5) \cup \P(3,5) \cup \P(3,4) \cup \P(2,5) \cup \P(2,4) \cup \P(2,3)\cup \P(1,5) \cup \P(1,4) \cup \P(1,3) \cup \P(1,2)$.
\end{enumerate}
\end{enumerate}
\end{thm}

\begin{proof}
(1) By \lemref{lem.clp1}, there exists a graded ($\pm 1$)-skew polynomial algebra $S'$ such that $A \cong S'/(f)$ and the point scheme $E'$ of $S'$ is $\P^3,  \P(1,2,4) \cup \P(1,2,3) \cup \P(3,4)$, or $\bigcup_{1\leq i<j\leq 4} \P(i,j)$. (Note that $E \cong E'$.)
By \thmref{thm.gst} (2), (3), we only consider the case $E'=\P(1,2,4) \cup \P(1,2,3) \cup \P(3,4)$.
In this case, 
\[\ep_{41}\ep_{12}\ep_{24}=1,\; \ep_{41}\ep_{13}\ep_{34}=-1,\; \ep_{42}\ep_{23}\ep_{34}=-1\]
(see (4\rnum{3}) in the proof of \lemref{lem.clp1}), so $C(S'/(f))$ is isomorphic to
\[ k\ang{t_1, t_2, t_3}/(t_1t_2+t_2t_1, t_1t_3-t_3t_1, t_2t_3-t_3t_2, t_i^2-1) \cong M_2(k)^2\]
by \lemref{lem.clc} (2). Thus we have $\uCM(A) \cong \uCM(S'/(f)) \cong \Db(\mod k^2)$ by \thmref{thm.sv}.

(2) By \lemref{lem.clp2}, there exists a graded ($\pm 1$)-skew polynomial algebra $S'$ such that $A \cong S'/(f)$
and the point scheme $E'$ of $S'$ is $\text{(5a)}, \dots, \text{(5f)}$, or (5g).
By \thmref{thm.gst} (1), (3), we only consider the cases (5b) to (5f).

If $E$ is (5b), then
\[\ep_{51}\ep_{12}\ep_{25}=1,\; \ep_{51}\ep_{13}\ep_{35}=1,\; \ep_{51}\ep_{14}\ep_{45}=-1,\;
  \ep_{52}\ep_{23}\ep_{35}=1,\; \ep_{52}\ep_{24}\ep_{45}=-1,\; \ep_{53}\ep_{34}\ep_{45}=-1, \]
(see (5\rnum{5}) in the proof of \lemref{lem.clp2}), so $C(S'/(f))$ is isomorphic to
\[ k\ang{t_1, t_2, t_3, t_4}/(t_1t_2+t_2t_1, t_1t_3+t_3t_1, t_1t_4-t_4t_1, t_2t_3+t_3t_2, t_2t_4-t_4t_2, t_3t_4-t_4t_3, t_i^2-1) \cong M_2(k)^4\]
by \lemref{lem.clc} (3). Thus we have $\uCM(A) \cong \uCM(S'/(f)) \cong \Db(\mod k^4)$ by \thmref{thm.sv}.

If $E$ is (5c), then 
\[\ep_{51}\ep_{12}\ep_{25}=1,\; \ep_{51}\ep_{13}\ep_{35}=-1,\; \ep_{51}\ep_{14}\ep_{45}=-1,\;
  \ep_{52}\ep_{23}\ep_{35}=-1,\; \ep_{52}\ep_{24}\ep_{45}=-1,\; \ep_{53}\ep_{34}\ep_{45}=1, \]
(see (5\rnum{8}) in the proof of \lemref{lem.clp2}), so $C(S'/(f))$ is isomorphic to
\[ k\ang{t_1, t_2, t_3, t_4}/(t_1t_2+t_2t_1, t_1t_3-t_3t_1, t_1t_4-t_4t_1, t_2t_3-t_3t_2, t_2t_4-t_4t_2, t_3t_4+t_4t_3, t_i^2-1) \cong M_4(k)\]
by \lemref{lem.clc} (4). Thus we have $\uCM(A) \cong \uCM(S'/(f)) \cong \Db(\mod k)$ by \thmref{thm.sv}.

If $E$ is (5d), then 
\[\ep_{51}\ep_{12}\ep_{25}=1,\; \ep_{51}\ep_{13}\ep_{35}=-1,\; \ep_{51}\ep_{14}\ep_{45}=1,\;
  \ep_{52}\ep_{23}\ep_{35}=-1,\; \ep_{52}\ep_{24}\ep_{45}=-1,\; \ep_{53}\ep_{34}\ep_{45}=1, \]
(see (5\rnum{6}) in the proof of \lemref{lem.clp2}), so $C(S'/(f))$ is isomorphic to
\[ k\ang{t_1, t_2, t_3, t_4}/(t_1t_2+t_2t_1, t_1t_3-t_3t_1, t_1t_4+t_4t_1, t_2t_3-t_3t_2, t_2t_4-t_4t_2, t_3t_4+t_4t_3, t_i^2-1) \cong M_4(k)\]
by \lemref{lem.clc} (5). Thus we have $\uCM(A) \cong \uCM(S'/(f)) \cong \Db(\mod k)$ by \thmref{thm.sv}.

If $E$ is (5e), then
\[\ep_{51}\ep_{12}\ep_{25}=1,\; \ep_{51}\ep_{13}\ep_{35}=1,\; \ep_{51}\ep_{14}\ep_{45}=-1,\;
  \ep_{52}\ep_{23}\ep_{35}=-1,\; \ep_{52}\ep_{24}\ep_{45}=-1,\; \ep_{53}\ep_{34}\ep_{45}=-1, \]
(see (5\rnum{9}) in the proof of \lemref{lem.clp2}), so $C(S'/(f))$ is isomorphic to
\[ k\ang{t_1, t_2, t_3, t_4}/(t_1t_2+t_2t_1, t_1t_3+t_3t_1, t_1t_4-t_4t_1, t_2t_3-t_3t_2, t_2t_4-t_4t_2, t_3t_4-t_4t_3, t_i^2-1) \cong M_2(k)^4\]
by \lemref{lem.clc} (6). Thus we have $\uCM(A) \cong \uCM(S'/(f)) \cong \Db(\mod k^4)$ by \thmref{thm.sv}.

If $E$ is (5f), then 
\[\ep_{51}\ep_{12}\ep_{25}=1,\; \ep_{51}\ep_{13}\ep_{35}=-1,\; \ep_{51}\ep_{14}\ep_{45}=-1,\;
  \ep_{52}\ep_{23}\ep_{35}=-1,\; \ep_{52}\ep_{24}\ep_{45}=-1,\; \ep_{53}\ep_{34}\ep_{45}=-1, \]
(see (5\rnum{10}) in the proof of \lemref{lem.clp2}), so $C(S'/(f))$ is isomorphic to
\[ k\ang{t_1, t_2, t_3, t_4}/(t_1t_2+t_2t_1, t_1t_3-t_3t_1, t_1t_4-t_4t_1, t_2t_3-t_3t_2, t_2t_4-t_4t_2, t_3t_4-t_4t_3, t_i^2-1) \cong M_2(k)^4\]
by \lemref{lem.clc} (7). Thus we have $\uCM(A) \cong \uCM(S'/(f)) \cong \Db(\mod k^4)$ by \thmref{thm.sv}.
\end{proof}

Let $\ell$ denote the number of irreducible components of $E$ which are isomorphic to $\P^1$, that is, the the number of irreducible components of the form $\P(i,j)$.
\corref{cor.gst} and \thmref{thm.clm} imply the following result which states that \conjref{conj.main} is true for $n\leq 5$.

\begin{thm}\label{thm.main}
Assume that $n\leq 5$.
\begin{enumerate}
\item If $n$ is odd, then $\ell\leq 10$ and
\begin{align*}
\ell =0 &\Longleftrightarrow \uCM(A) \cong \Db(\mod k),\\
0< \ell \leq 3 &\Longleftrightarrow \uCM(A) \cong \Db(\mod k^4),\\
3< \ell \leq 10 &\Longleftrightarrow \uCM(A) \cong \Db(\mod k^{16}).
\end{align*}
\item If $n$ is even, then $\ell\leq 6$ and
\begin{align*}
0\leq \ell \leq 1 &\Longleftrightarrow \uCM(A) \cong \Db(\mod k^2),\\
1< \ell \leq 6 &\Longleftrightarrow \uCM(A) \cong \Db(\mod k^8).
\end{align*}
\end{enumerate}
\end{thm}

At the end of paper, we collect some examples when $n=6$ as further evidence for \conjref{conj.main}.

\begin{ex}
\begin{enumerate}
\item Let $S = k\ang{x_1, \dots, x_6}/(x_ix_j-\ep_{ij}x_jx_i)$ with
\begin{align*}
&\ep_{12}=1, \;\ep_{13}=-1, \;\ep_{14}=1, \;\ep_{15}=-1, \;\ep_{16}=1, \;\ep_{23}=-1, \;\ep_{24}=-1, \;\ep_{25}=-1, \;\ep_{26}=1,\\
&\ep_{34}=1, \;\ep_{35}=-1, \;\ep_{36}=1, \;\ep_{45}=-1, \;\ep_{46}=1, \;\ep_{56}=1.
\end{align*}
Then the point scheme of $S$ is $\P(3,4,5) \cup \P(2,3,4) \cup \P(1,4,5) \cup \P(1,2,5) \cup \P(1,2,3) \cup \P(3,4,6) \cup \P(1,4,6) \cup \P(1,2,6) \cup \P(5,6)$,
so $\ell =1$. On the other hand, one can check that $C(A) \cong M_4(k)^2$, so we have $\uCM(A) \cong \Db(\mod k^2)$.

\item Let $S = k\ang{x_1, \dots, x_6}/(x_ix_j-\ep_{ij}x_jx_i)$ with
\begin{align*}
&\ep_{12}=1, \;\ep_{13}=-1, \;\ep_{14}=-1, \;\ep_{15}=-1, \;\ep_{16}=1, \;\ep_{23}=1, \;\ep_{24}=-1, \;\ep_{25}=-1, \;\ep_{26}=1,\\
&\ep_{34}=-1, \;\ep_{35}=-1, \;\ep_{36}=1, \;\ep_{45}=1, \;\ep_{46}=1, \;\ep_{56}=1.
\end{align*}
Then the point scheme of $S$ is $\P(2,3,4,5) \cup \P(1,2,4,5) \cup \P(2,3,6) \cup \P(1,2,6) \cup \P(4,5,6) \cup \P(1,3)$,
so $\ell =1$. On the other hand, one can check that $C(A) \cong M_4(k)^2$, so we have $\uCM(A) \cong \Db(\mod k^2)$.

\item Let $S = k\ang{x_1, \dots, x_6}/(x_ix_j-\ep_{ij}x_jx_i)$ with
\begin{align*}
&\ep_{12}=1, \;\ep_{13}=-1, \;\ep_{14}=-1, \;\ep_{15}=-1, \;\ep_{16}=1, \;\ep_{23}=1, \;\ep_{24}=-1, \;\ep_{25}=-1, \;\ep_{26}=1,\\
&\ep_{34}=-1, \;\ep_{35}=-1, \;\ep_{36}=1, \;\ep_{45}=-1, \;\ep_{46}=1, \;\ep_{56}=1.
\end{align*}
Then the point scheme of $S$ is $\P(2,3,5) \cup \P(2,3,4) \cup \P(1,2,5) \cup \P(1,2,4) \cup \P(1,2,6) \cup \P(2,3,6)\cup \P(4,5) \cup \P(1,3) \cup \P(4,6)\cup \P(5,6)$,
so $\ell =4$. On the other hand, one can check that $C(A) \cong M_2(k)^8$, so we have $\uCM(A) \cong \Db(\mod k^8)$.

\item Let $S = k\ang{x_1, \dots, x_6}/(x_ix_j-\ep_{ij}x_jx_i)$ with
\begin{align*}
&\ep_{12}=1, \;\ep_{13}=-1, \;\ep_{14}=-1, \;\ep_{15}=-1, \;\ep_{16}=1, \;\ep_{23}=-1, \;\ep_{24}=-1, \;\ep_{25}=-1, \;\ep_{26}=1,\\
&\ep_{34}=-1, \;\ep_{35}=-1, \;\ep_{36}=1, \;\ep_{45}=-1, \;\ep_{46}=1, \;\ep_{56}=1.
\end{align*}
Then the point scheme of $S$ is $\P(1,2,5) \cup \P(1,2,4) \cup \P(1,2,3) \cup \P(1,2,6) \cup \P(4,5) \cup \P(3,5) \cup \P(3,4) \cup \P(4,6) \cup \P(3,6)\cup \P(5,6)$,
so $\ell =6$. On the other hand, one can check that $C(A) \cong M_2(k)^8$, so we have $\uCM(A) \cong \Db(\mod k^8)$.
\end{enumerate}
\end{ex}

\section*{Acknowledgment} 
The author thanks the referee for a careful reading of the manuscript and helpful comments.

\end{document}